\documentclass[10pt,oneside,a4paper]{amsart}
\numberwithin{equation}{section}
\usepackage{amsmath,amsfonts,amsthm, amssymb}
\usepackage{array}
\usepackage[all,textures]{xy}
\usepackage{portland}

\theoremstyle{plain}
\newtheorem{theorem}{Theorem}[section]
\newtheorem{proposition}[theorem]{Proposition}
\newtheorem{lemma}[theorem]{Lemma}
\newtheorem{corollary}[theorem]{Corollary}

\theoremstyle{definition}
\newtheorem{definition}[theorem]{Definition}

\theoremstyle{remark}

\newtheorem{remark}[theorem]{Remark}

\renewcommand{\lim}{\mathrm{lim}}

\newcommand{\Ext}{\mathrm{Ext}}

\newcommand{\Hom}{\mathrm{Hom}}

\newcommand{\RHom}{\mathrm{RHom}}

\newcommand{\op}{^{\mathrm{op}}}
\newcommand{\Ob}{\mathrm{Ob}}

\newcommand{\AAA}{\mathfrak{a}}
\newcommand{\BBB}{\mathfrak{b}}

\newcommand{\RRR}{\mathfrak{r}}

\newcommand{\TTT}{\mathfrak{t}}

\newcommand{\CC}{\mathbf{C}}

\newcommand{\Mod}{\ensuremath{\mathsf{Mod}} }

\newcommand{\Pre}{\ensuremath{\mathsf{Pr}} }

\newcommand{\Cat}{\ensuremath{\mathsf{Cat}} }

\newcommand{\lra}{\longrightarrow}

\newcommand{\aaa}{\ensuremath{\mathcal{A}}}
\newcommand{\bbb}{\ensuremath{\mathcal{B}}}

\newcommand{\uuu}{\ensuremath{\mathcal{U}}}

\SelectTips{cm}{11}

\title{A Hochschild Cohomology Comparison Theorem for prestacks}
\author{Wendy Lowen} 
\address[Wendy Lowen]{Departement Wiskunde-Informatica, Middelheimcampus,
Middelheimlaan 1,
2020 Antwerp, Belgium}
\email{wendy.lowen@ua.ac.be}
\author{Michel Van den Bergh}
\address[Michel Van den Bergh]{Departement WNI, Hasselt University, Agoralaan, 3590 Diepenbeek, Belgium}
\email{michel.vandenbergh@uhasselt.be}
\thanks{The first author is postdoctoral fellow with the Fund of Scientific Research Flanders (FWO)}
\thanks{The second author is a director of research at the FWO}
 \keywords{Hochschild cohomology, fibered categories, Special Cohomology Comparison Theorem}
\subjclass{16E40,18D30}
\begin{document}
\maketitle

\begin{abstract}
  We generalize and clarify Gerstenhaber and Schack's ``Special
  Cohomology Comparison Theorem''.  More specifically we obtain a
  fully faithful functor between the derived categories of bimodules
  over a prestack over a small category $\mathcal{U}$ and the derived
  category of bimodules over its corresponding fibered category. In
  contrast to Gerstenhaber and Schack we do not have to assume that
  $\mathcal{U}$ is a poset.
\end{abstract}

\def\Bimod{\mathsf{Bimod}}
\section{Introduction}
Throughout $k$ is a commutative base ring.
In \cite{gerstenhaberschack, gerstenhaberschack1,
  gerstenhaberschack2} Gerstenhaber and Schack study deformation
theory and Hochschild cohomology of presheaves of algebras. For a
presheaf $\aaa$ of $k$-algebras on a small category $\uuu$, the
corresponding Hochschild cohomology is defined as\footnote{For applications in deformation
theory one would only use this definition if $\mathcal{A}$ is flat over
$k$. The ``correct'' definition is obtained by replacing $\mathcal{A}$ first by a
suitable $k$-flat resolution.}
\begin{equation}\label{ref-1.1-0}
\operatorname{HH}^n(\aaa) = \Ext^n_{\Bimod(\aaa)}(\aaa, \aaa)
\end{equation}
To $\aaa$ Gerstenhaber and Schack associate a single algebra $\mathcal{A}!$, and
a functor
\[
(-)!: \Bimod(\aaa) \rightarrow \Bimod( \mathcal{A}!),
\] 
which sends $\aaa$ to $\aaa!$ and preserves $\Ext$. 
It follows
in particular that $\operatorname{HH}^n(\aaa) \cong
\operatorname{HH}^n(\mathcal{A!})$.  As $(-)!$ does not preserve
injectives nor projectives the fact that it preserves $\Ext$ is not at
all a tautology

\medskip

In fact the construction of $\mathcal{A}!$ and the proof of preservation of
$\Ext$ are rather difficult and proceed in several steps. The first
step covers the case that $\uuu$ is a \emph{poset}. In that case
$\mathcal{A}!$ is simply 
$\prod_{V\in \uuu} \bigoplus_{V\leq U}
\aaa(V)$. This part of the
construction is the so-called \emph{Special Cohomology Comparison
  Theorem} (SCCT). It is stated and proved in
\cite{gerstenhaberschack2}.

To cover the general case Gerstenhaber and Schack doubly subdivide
$\uuu$ which transforms it into a poset.  The \emph{General Cohomology
  Comparison Theorem} (GCCT) states that this subdivision preserves
$\Ext$. The statement of this theorem is given in~\cite{gerstenhaberschack1}, but the proof has not been published.

In this paper we will be concerned with the SCCT. Our methods may also be
used to give new insight on the GCCT but this will be covered elsewhere. 

\medskip

We now discuss our main result (in a less general setting than in the
body of the paper for the purposes of exposition). To $\aaa$ we may
associate a ``$\uuu$-graded category'' (see \cite{lowen8} or
\S\ref{ref-2.2-3} below) 
which has an associated category of ``bimodules''
$\mathsf{Bimod}_{\uuu}(\AAA)$.  Our generalization of the SCCT is the
following
\begin{theorem} \label{ref-1.1-1} (A restricted version of Theorem \ref{ref-4.1-40})
There is a natural
functor
\[
\Pi^{\ast}: \Bimod(\aaa) \rightarrow \mathsf{Bimod}_{\uuu}(\AAA)
\]
which has the property $\Pi^\ast(\aaa)=\AAA$ and which induces a fully
faithful functor between the corresponding derived categories. 
\end{theorem}
In case $\uuu$ is a poset this theorem quickly yields the SCCT in the
version of Gerstenhaber and Schack (see \S \ref{ref-5-42}). 
Despite the fact that our result is more general our proof seems
more direct than the one by Gerstenhaber and Schack as we are able to leverage some basic properties
of natural systems \cite{baueswirsching}.

Our main application of Theorem \ref{ref-1.1-1} is that combined with
\cite[Thm 3.27]{lowen8} it implies that the Hochschild cohomology of
$\mathcal{A}$ controls the deformation theory of $\mathcal{A}$ as a
$k$-linear prestack\footnote{$k$-linear prestacks over $\uuu$ were
  called ``pseudofunctors'' in \cite{lowen8}.} (but not as a presheaf
of $k$-algebras!). If fact from Theorem~\ref{ref-4.1-40} below it
follows that a similar result is true if $\mathcal{A}$ is itself a
$k$-linear prestack.

In an appendix we describe the essential image of $\Pi^\ast$ as the
objects inverting certain maps between projectives.  Theorem
\ref{ref-1.1-1} may be translated into saying that $\Pi^\ast$ is
obtained from a certain \emph{stably flat} universal localization
\cite{NR} of linear categories. In particular the K-theoretic results
of \cite{neeman2,NR} apply.

Finally we mention that Theorem \ref{ref-1.1-1}
(or rather its generalization  Theorem \ref{ref-4.1-40}) is a key ingredient in \cite{lowenvandenbergh3}.

\section{Prestacks and graded categories}
In this section we quickly recall the relation between fibered categories
and prestacks in the $k$-linear setting. For full details we refer to
\cite{lowen8,vistoli2}.
\subsection{The classical formalism}
We first recall the classical theory when there is no additive
structure.  Let $\mathcal{U}$, $\AAA$, $\phi$ be respectively a small category, an
arbitrary category and
a functor $\phi:\AAA\rightarrow \uuu$. For
$A,B\in \AAA$ and for $f:\phi(A)\rightarrow \phi(B)$ in $\mathcal{U}$ put
$\AAA_f(A,B)=\phi^{-1}(f)$.  For $U\in \Ob(\uuu)$, we define a
subcategory $\AAA_U$ of $\AAA$ as follows: 
$\Ob(\AAA_U)=\phi^{-1}(U)$ and for $A,B\in \phi^{-1}(U)$ we put
$\AAA_U(A,B)=\AAA_{1_U}(A,B)$.

Instead of specifying the functor $\phi$ we may just as well specify 
the (possibly big) sets $\Ob(\AAA_U)$ and  for $A\in \Ob(\AAA_U)$, $B\in \Ob(\AAA_V)$ the decompositions
\[
\frak{a}(A,B)=\coprod_{f:U\rightarrow V} \frak{a}_f(A,B)
\]
This is what we will do in the sequel. 

\medskip

If $f:U\rightarrow V$ then an arrow $\delta\in \frak{a}_{f}(A,B)$,
$A\in \frak{a}_U$, $B\in \frak{a}_V$ is called \emph{cartesian} if
left composing with $\delta$ defines an isomorphism $\frak{a}_{g}(C,A)
\rightarrow \frak{a}_{fg}(C,B)$ for all $W\in \uuu$, $C\in
\frak{a}_w$, $g:W\rightarrow U$. Given $f,B$ a cartesian arrow is
necessarily unique up to unique isomorphism.

We say that $\frak{a}$ is \emph{fibered} if for any $f:U\rightarrow V$
in $\uuu$ and any $B\in \frak{a}_V$ there exists a cartesian arrow
$\delta_{f,B}:A\rightarrow B$.

Given a fibered $\uuu$-category, the choice of cartesian morphisms
$\delta_{f,B}\in \AAA_f(A,B)$ for every $f,B$ is called a
\emph{cleavage}.  We will always choose a \emph{normalized cleavage}, i.e.\
one in which $\delta_{1_V,B}=\operatorname{Id}_B$. If we have chosen a
cleavage then the domain of $\delta_{f,B}$ is denoted by $f^\ast B$.

In this way for every $B$ we obtain a functor
$f^\ast:\frak{a}_V\rightarrow \frak{a}_U$. For compositions $W\xrightarrow{g}
U\xrightarrow{f} V$ we obtain natural isomorphisms $(fg)^\ast\cong
g^\ast f^\ast$ which satisfy the usual compatibility for triple
compositions. In other words $U\mapsto \frak{a}_U$ defines a
pseudo-functor $\AAA:\uuu\rightarrow \operatorname{Cat}$. This pseudo-functor
satisfies $\AAA(1_V)=\operatorname{Id}_{\AAA_V}$. We will call such
a pseudo-functor \emph{normalized}. 

We now have functors of 2-categories
\begin{multline}
\label{ref-2.1-2}
\{\text{fibered $\uuu$-categories}\}\leftarrow \{\text{fibered $\uuu$-categories with a normalized cleavage} \}\\
\rightarrow \{\text{normalized pseudo-functors
$\uuu\rightarrow \operatorname{Cat}$}\}
\end{multline} 
where the first one is the forgetful functor and the second one is the
construction outlined in the above paragraphs. The above discussion
shows that the first functor is an equivalence and one easily verifies
that the second one is an isomorphism. The inverse functor associates
to a normalized pseudo-functor $\aaa:\uuu\rightarrow
\operatorname{Cat}$ the fibered category $\frak{a}$ such that
$\frak{a}_U=\aaa(U)$ and 
\[
\frak{a}_f(A,B)=\aaa(U)(A,f^\ast B)
\]
for $f:U\rightarrow V$, $A\in \AAA_U$, $B\in \AAA_V$. A normalized
cleavage is given by defining $\delta_{f,B}$ as the identity map in
$\frak{a}_f(f^\ast B,B)=\aaa(U)(f^\ast B,f^\ast B)$.

\subsection{Additive structure}
\label{ref-2.2-3}

\let\cal\mathcal

Let $\frak{a}\rightarrow \mathcal{U}$ be as above. Following
\cite{lowen8} we say that $\frak{a}$ is ($k$-linear)
\emph{$\mathcal{U}$-graded} if the sets $\frak{a}_f(A,B)$ are equipped
with the structure of a $k$-module such that the compositions
$\frak{a}_g(B,C)\times\frak{a}_f(A,B)\rightarrow \frak{a}_{gf}(A,C)$
are $k$-bilinear. If $\mathfrak{a}$ is fibered and we run this
additional structure through \eqref{ref-2.1-2} we find that
$\frak{a}$ now corresponds to a normalized pseudo-functor
$\cal{U}\rightarrow \Cat(k)$ (where $\Cat(k)$ stands for the $2$-category
of $k$-linear categories) and that this correspondence is reversible.

\medskip 

Below we will call a normalized pseudo-functor $\cal{U}\rightarrow
\Cat(k)$ a \emph{$k$-linear prestack} on $\cal{U}$. 
\begin{remark} If we equip $\cal{U}$ with the
  trivial Grothendieck topology then this use of the word ``prestack''
  is consistent with standard terminology as the usual gluing
  condition on maps is empty. Since the same is true for
 objects we might even have talked about
  stacks instead of prestacks.
\end{remark}

\section{Module and bimodule categories}\label{ref-3-4}
In this section we introduce a number of (bi)module categories and
relate them. The only non-formal result is Lemma \ref{ref-3.9-26}. 
\subsection{Modules over a $k$-linear prestack}
Recall that if $\frak{l}$ is a $k$-linear category then a (right)
$\frak{l}$-module is by definition a $k$-linear contravariant functor
$\frak{l}\rightarrow \Mod(k)$. This may be generalized to $k$-linear prestacks. 
To do this consider the constant presheaf $\underline{\Mod}(k)$ as a
prestack on $\uuu$.
\begin{definition}
  Let $\aaa$ be a $k$-linear prestack on $\uuu$. An \emph{$\aaa$-module} is
  a morphism of prestacks $M: \aaa^{\op} \lra
  \underline{\Mod}(k)$. More concretely, an $\aaa$-module consists of
  the following data:
\begin{itemize}
\item for every $U \in \uuu$, a $\aaa(U)$-module $M^U$;
\item for every $u: V \lra U$, a morphism 
of $\aaa(U)$-modules $\rho^u: M^U \lra M^Vu^\ast$;
\end{itemize}
such that the following additional compatibilities hold:
\begin{itemize}
\item for every $u: V \lra U$, $v: W \lra V$, $\rho^{uv}$ equals the canonical composition
$$M^U \lra M^Vu^\ast \lra M^W v^\ast u^\ast \lra  M^W (uv)^\ast;$$
\item for every $U \in \uuu$, $\rho^U: M^U \lra M^U1^{\ast}=M^U$ is the identity.
\end{itemize}
A \emph{morphism} of $\aaa$-modules $\varphi: M \lra N$ consists of morphisms
$$\varphi^U: M^U \lra N^U$$
commuting with the $\rho^u$. 
\end{definition}
Modules over a $k$-linear prestack $\aaa$ and their morphisms constitute an
abelian category ${\Mod}(\aaa)$.

\subsection{Bimodules over $k$-linear prestacks}
Let $\mathcal{A}$, $\mathcal{B}$ be two $k$-linear prestacks. By
definition an $\mathcal{A}$-$\mathcal{B}$-bimodule is a module over
$\mathcal{A}^{\operatorname{op}}\otimes\mathcal{B}$. More concretely a
$\mathcal{A}$-$\mathcal{B}$-bimodule $M$ is given by $k$-modules 
\[
M^U(B,A)
\] for
$U\in \Ob(\mathcal{U})$, $A\in \Ob(\mathcal{A}_U)$, $B\in \Ob(\mathcal{B}_U)$
which vary covariantly in $A$ and contravariantly in $B$ together with
compatible restriction morphisms
\[
\rho^u(B,A):M^U(B,A)\rightarrow M^V(u^\ast B, u^\ast A)
\]
for $u:V\rightarrow U$ in $\mathcal{U}$.  

The abelian category of
$\mathcal{A}$-$\mathcal{B}$-bimodules is denoted by
$\Bimod(\mathcal{A},\mathcal{B})$.

\subsection{Bimodules over a graded category}\label{ref-3.3-5}

$\mathcal{U}$-graded categories are mild generalizations of linear
categories. In particular they admit a natural notion of bimodule.

\begin{definition} \cite[Definition 2.9]{lowen8} Let $\AAA$ and $\BBB$
  be $\uuu$-graded categories. An $\AAA$-$\BBB$-\emph{bimodule} $M$
  consists of $k$-modules $$M_u(B,A)$$ for $u: V \lra U$, $B \in
  \Ob(\BBB_V)$, $A \in \Ob(\AAA_U)$ and ``multiplication'' morphisms
\begin{equation}
\label{ref-3.1-6}
\rho: \AAA_w(A,A') \otimes M_u(B,A) \otimes \BBB_v(B',B) \lra M_{wuv}(B',A')
\end{equation}
satisfying the natural associativity and identity axioms. 
\end{definition}
The abelian category of $\AAA$-$\BBB$-bimodules is denoted by
$\Bimod_{\mathcal{U}}(\AAA,\BBB)$.

\subsection{Functors between bimodule categories}\label{ref-3.4-7}
Let $\aaa$ and $\bbb$ be $k$-linear prestacks on $\mathcal{U}$ with associated
fibered graded categories $\AAA$ and $\BBB$ (see \S\ref{ref-2.2-3}).

There is a natural functor
\[
\Pi^{\ast}: {\Bimod}(\aaa, \bbb) \lra \mathsf{Bimod}_{\uuu}(\AAA,\BBB)
\]
defined by
\begin{equation}
\label{ref-3.2-8}
(\Pi^{\ast}M)_{u}(B, A) = M^V(B, u^{\ast}A)
\end{equation}
for $u:V\rightarrow U$, $A\in \Ob(\mathcal{A}(U))=\Ob(\mathfrak{a}_U)$, $B\in
\Ob(\mathcal{B}(V))=\Ob(\mathfrak{b}_V)$. 

\subsection{Fibered bimodules}
In this section we identify the essential image of $\Pi^\ast$.
\begin{definition}
\label{ref-3.3-9}
  Let $\AAA$ and $\BBB$ be fibered $\uuu$-graded categories. An
  $\AAA$-$\BBB$-bimodule $M$ is called \emph{fibered} if for one (and
  hence for every) cartesian morphism $\delta \in \AAA_u(A,A')$ with
  $u:V\rightarrow U$, $A\in \AAA_V$, $A'\in \AAA_U$ we have for every
 $v: W \lra V$, $B \in \BBB_W$ that the map
\[
 M_v(B,A) \lra M_{uv}(B,A')
\]
given by left multiplication (cfr \eqref{ref-3.1-6}) with $\delta$ is an
isomorphism.
\end{definition}

Let $$\mathsf{Bimod}^{\mathrm{fib}}_{\uuu}(\AAA,\BBB) \subseteq \mathsf{Bimod}_{\uuu}(\AAA,\BBB)$$ denote the full subcategory of fibered bimodules.

\begin{proposition} \label{ref-3.4-10}
  Assume that $\AAA$ and $\BBB$ are obtained from $k$-linear prestacks
  $\aaa$ and $\bbb$. Then  the functor $\Pi^{\ast}$ induces an equivalence
  of categories
\[
{\Bimod}(\aaa , \bbb) \lra \mathsf{Bimod}^{\mathrm{fib}}_{\uuu}(\AAA,\BBB).
\]
In particular $\Pi^\ast$ is fully faithful.
\end{proposition}
\begin{proof}
  Let $N\in {\Bimod}(\aaa , \bbb)$. We first show that $\Pi^\ast N$ is a
  fibered bimodule. With the notation of Definition \ref{ref-3.3-9} we have
to show that the composition map
\[
(\Pi^\ast N)_v(B,A) \lra (\Pi^\ast N)_{uv}(B,A')
\]
is an isomorphism. We may choose $A=u^\ast A'$ and $\delta=\delta_{u,A'}$. 

By \eqref{ref-3.2-8} we have 
\begin{align*}
(\Pi^\ast N)_v(B,A)&=N^W(B,v^\ast
A)=N^W(B,v^\ast u^\ast A')\\ 
(\Pi^\ast N)_{uv}(B,A')&=N^W(B,(uv)^\ast A')
\end{align*}
Making explicit the various definitions we see that the composition map is
derived from the isomorphism $v^\ast u^\ast =(uv)^\ast$ and hence is itself
an isomorphism. Thus $\Pi^\ast N$ is fibered.

\medskip

Conversely assume that $M$ is a fibered $\AAA$-$\BBB$ bimodule. Then
we may define an $\mathcal{A}$-$\mathcal{B}$ bimodule $\Pi_\ast M$ via
\[
(\Pi_\ast M)^U(B,A)=M_{1_U}(B, A)
\]
for $A\in \Ob(\mathcal{A}(U))=\Ob(\mathfrak{a}_U)$, $B\in
\Ob(\mathcal{B}(U))=\Ob(\mathfrak{b}_U)$ where the restriction maps
\begin{multline*}
M_{1_U}(B,A)\overset{\text{def}}{=}(\Pi_\ast M)^U(B,A)\\
\xrightarrow{\rho^u} (\Pi_\ast M)^V(u^\ast B,u^\ast A)\overset{\text{def}}{=}M_{1_V}(u^\ast B,u^\ast A)\overset{\text{fib.}}{=}M_u(u^\ast B,A)
\end{multline*}
for $u:V\rightarrow U$ are given by by right multiplication with
the cartesian arrow  $\delta_{u,B}\in \AAA_u(u^\ast B, B)$. 

\medskip

It is easy to see that $\Pi^\ast$ and $\Pi_\ast$ define quasi-inverse functors
between $\Bimod(\aaa,\bbb)$ and $\Bimod_{\cal{U}}^{\text{fib}}(\AAA,\BBB)$. 
\end{proof}
In the sequel we will work with the category
$\Bimod_{\cal{U}}^{\text{fib}}(\AAA,\BBB)$ instead of $\Bimod(\aaa,\bbb)$.
The previous proposition shows this is equivalent. We will denote the inclusion
$ \Bimod_{\cal{U}}^{\text{fib}}(\AAA,\BBB)\hookrightarrow  \Bimod_{\cal{U}}(\AAA,\BBB)$ by $I$. 
\subsection{Relation with presheaves and natural systems on $\mathcal{U}$}
\label{ref-3.6-11}
Let $\AAA$ and $\BBB$ be fibered $\mathcal{U}$-graded
categories equipped with a normalized cleavage. Assume that $M$ is a fibered $\AAA$-$\BBB$-bimodule. Pick
$W\in \mathcal{U}$, $A\in \AAA_W$, $B\in \BBB_W$. Then $M$ induces a
presheaf $\Psi^\ast_{A,B}(M)$ of $k$-modules over $\mathcal{U}/W$ as
follows: for $g:V\rightarrow W$ in $\mathcal{U}/W$ put
\begin{equation}
\label{ref-3.3-12}
\Psi^\ast_{A,B}(M)(g)=M_{1_V}(g^\ast B,g^\ast A)
\end{equation}
The restriction morphisms are obtained from the fact that $M$ is obtained from
a bimodule over the prestacks corresponding to $\AAA$, $\BBB$ (by Proposition
\ref{ref-3.4-10}). 

Concretely for $v:V'\rightarrow V$  the restriction
map
\begin{multline*}
M_{1_V}(g^\ast B,g^\ast A)\overset{\text{def}}{=} \Psi^\ast_{A,B}(M)(g)\\
\xrightarrow{\rho^v}
\Psi^\ast_{A,B}(M)(gv)\overset{\text{def}}{=}M_{1_{V'}}((gv)^\ast B,(gv)^\ast A)
\overset{\text{fib.}}{=} M_{v}(v^\ast g^\ast B,g^\ast A)
\end{multline*}
is given by right multiplying with the cartesian arrow $\delta_{v,g^\ast B}\in
\mathfrak{b}_v(v^\ast g^\ast B,g^\ast B)$. 

\medskip

We think of
$\Psi^\ast_{A,B}(M)$ as a $(A,W,B)$-local version of $M$.
For use below we present a generalization of this construction to the
case that $M$ is not necessarily fibered. 

\def\Nat{\mathsf{Nat}}

A \emph{natural system} of $k$-modules on $\uuu$ in the sense of
\cite{baueswirsching} is by definition a functor 

\begin{equation}
\label{ref-3.4-13}
\mathsf{Fact}(\uuu)
\lra \Mod(k)
\end{equation} where $\mathsf{Fact}(\uuu)$ is the category with as
objects the morphisms $u: V \lra U$ of $\uuu$, and morphisms from $u$
to $u'$ given by diagrams
\[
\xymatrix{ {V} \ar[r]^u & {U} \ar[d]^p \\ {V'} \ar[u]^q \ar[r]_{u'} & {U'.}}
\]
Natural systems on $\uuu$ constitute a category $\mathsf{Nat}(\uuu)$.

Sending an arrow to its domain defines a functor
\[
\mathsf{Fact}(\uuu)\lra \mathcal{U}^{\op}
\]
Hence from \eqref{ref-3.4-13} we obtain a corresponding functor
\[
I:\Pre(\mathcal{U})\lra \mathsf{Nat}(\uuu)
\]
where $\Pre(\mathcal{U})$ is the category of $k$-linear presheaves on
$\uuu$.  Concretely for a presheaf $F$ and an arrow $u:V\rightarrow U$ we have
\[
(I\!F)(u)=F(V)
\]
Clearly $I$ is fully faithful.

\medskip

We define a natural system $\Phi^\ast_{A,B}(M)$ on $\mathcal{U}/W$ via
\[
\Phi^\ast_{A,B}(M)(u)=M_u((fu)^\ast B, f^\ast A)
\]
where $V\xrightarrow{u} U\xrightarrow{f} W$ represents an object in
$\mathsf{Fact}(\mathcal{U}/W)$ 

Let us check that this has the right functoriality property.  Consider
the following morphism in $\mathsf{Fact}(\mathcal{U}/W)$:
\[
\xymatrix{ {V} \ar[r]^u & {U} \ar[rd]^f \ar[d]^p \\ {V'} \ar[u]^q \ar[r]_{u'} & {U'} \ar[r]_{f'} & {W.}}
\]
We have to produce a map
\[
M_u((fu)^\ast B, f^\ast A)  \overset{\text{def}}{=} \Phi^\ast_{A,B}(M)(u)
\lra \Phi^\ast_{A,B}(M)(u')\overset{\text{def}}{=}M_{u'}((f'u')^\ast B, f^{\prime \ast} A)
\]
We have
\begin{align*}
M_u((fu)^\ast B, f^\ast A)&\cong M_u( (fu)^\ast B, p^\ast f^{\prime \ast} A)\\
M_{u'}((f'u')^\ast B, f^{\prime\ast} A)&\cong M_{puq}(q^\ast (fu)^\ast  B, 
f^{\prime\ast} A)
\end{align*}
The required map is given by left multiplying with 
$\delta_{p,f^{\prime \ast}A}\in \mathfrak{a}_p(p^\ast f^{\prime \ast} A,f^{\prime \ast} A )$
and right multiplying with $\delta_{q,(uf)^{\ast}B}\in \mathfrak{b}_v(v^\ast
(uf)^{\ast}B ,(uf)^{\ast}B) $.
\begin{proposition} 
\label{ref-3.5-14} Let $\frak{a}$, $\frak{b}$ be fibered $\mathcal{U}$-graded categories and let $W\in \mathcal{U}$, $A\in \AAA_W$, $B\in \BBB_W$. 
The following diagram is commutative
\[
\xymatrix{
\Bimod^{\mathrm{fib}}(\AAA,\BBB)\ar[d]_{\Psi^\ast_{A,B}} \ar[r]^I &
\Bimod(\AAA,\BBB\ar[d]^{\Phi^\ast_{A,B}})\\
\Pre(\uuu/W) \ar[r]_I & \Nat(\uuu/W)
}
\]
\end{proposition}
\begin{proof} Easy.
\end{proof}

\begin{remark}\label{remnatbimod}
Natural systems are a special case of bimodules over a fibered $\uuu$-graded category. More precisely, putting $\aaa = \bbb = \underline{k}$, the constant presheaf on $\uuu$, and $\mathfrak{k}$ the associated graded category, then natural systems on $\uuu$ are noting but $\mathfrak{k}$-$\mathfrak{k}$-bimodules.
\end{remark}

\subsection{Projective bimodules}
\label{ref-3.7-15}
Let $\AAA$, $\BBB$ be $\uuu$-graded categories.  %For use below we give
%a quite detailed description of the projective objects in
%$\Bimod^{\text{fib}}_{\mathcal{U}}(\AAA,\BBB)$ and
%$\Bimod_{\mathcal{U}}(\AAA,\BBB)$.
%
%\medskip
%
For
$u:V\rightarrow U$, $B\in \BBB_V$, $A\in \AAA_U$ the exact functor
\[
\Bimod_{\cal{U}}(\AAA,\BBB)\lra \Mod(k):M\mapsto M_u(B,A)
\]
is representable by a projective object which we denote
by $P_{B,u,A}$. Concretely for $u':V'\rightarrow U'$ in $\mathcal{U}$ and
$B\in \BBB_{V'}$, $A\in \AAA_{U'}$ we have
\begin{equation}
\label{ref-3.5-16}
(P_{B,u,A})_{u'}(B',A')=\bigoplus_{puq=u'} \BBB_q(B',B)\otimes \AAA_p(A,A')
\end{equation}
An element 
\[
b\otimes a \in \BBB_q(B',B)\otimes \AAA_p(A,A')
\]
may be interpreted in several equivalent ways.
\begin{itemize}
\item 
If $x\in M_{u}(B,A)$ then the corresponding map
\[
P_{B,u,A}\lra M
\]
sends $b\otimes a$ to $bxa$. 
\item 
If we view $b\otimes a$ as a map
\[
P_{B',u',A'}\lra P_{B,u,A}
\]
corresponding to $\operatorname{Id}_B\otimes \operatorname{Id}_A\in
(P_{B,u,A})_u(B,A)$ then the corresponding natural transformation from
the functor
\[
M\mapsto M_{u}(B,A)
\]
to the functor
\[
M\mapsto M_{u'}(B',A')
\]
is given by $m\mapsto amb$. 
\end{itemize}
From the second interpretation we obtain that if we have maps
\[
P_{B'',u'',A''}\xrightarrow{b'\otimes a'} P_{B',u',A'}\xrightarrow{b\otimes a}
P_{B,u',A}
\]
then the composition is given by $bb'\otimes a'a$.

For $u:V\rightarrow U$, $q:V'\rightarrow V$, $p:U\rightarrow U'$,
$B\in \AAA_V$,  $A\in \AAA_{U}$, $A'\in \AAA_{U'}$
we have canonical
maps
\begin{equation}
\label{ref-3.6-17}
\begin{aligned}
\delta^r_q:&P_{q^\ast B,uq,A}\lra P_{B,u,A}\\
\delta^l_p:&P_{B,pu,A'}\lra P_{B,u,p^\ast A'}
\end{aligned}
\end{equation}
which are respectively associated to
\begin{align*}
  \delta_{q,B}\otimes \operatorname{Id}_A&\in \BBB_q(q^\ast B,B)\otimes \AAA_{1_{U}}(A,A)\\
  \operatorname{Id}_B\otimes \delta_{p,A'}&\in
  \BBB_{1_{V}}(B,B)\otimes
  \AAA_{p}(p^\ast A',A')
\end{align*}
These canonical maps will play an important role below. 

\medskip

Now let $W\in \mathcal{U}$, $A\in \AAA_W$, $B\in \BBB_W$. 
The functors $\Phi^\ast_{A,B}$, $\Psi^\ast_{A,B}$ introduced in \S\ref{ref-3.6-11}
are functors between Grothendieck categories commuting with
products. Hence they have left adjoints which we denote respectively
by $\Phi_{A,B,!}$ and $\Psi_{A,B,!}$. Since these functors have exact
right adjoints they preserve projectives.
\begin{lemma}
\label{ref-3.6-18}
 For $V\xrightarrow{u} U\xrightarrow{f} W$ in $\mathcal{U}/W$ let $P_u$ be the projective natural
system on $\mathcal{U}/W$ given by
\[
P_u=k \mathsf{Fact}(\mathcal{U}/W)(u,-)
\]
\begin{enumerate}
\item
We have
\[
\Phi_{A,B,!}(P_u)=P_{(fu)^\ast B,u, f^\ast A}
\]
\item
The morphism $P_{u'}\rightarrow P_u$ corresponding to the morphism
$u\rightarrow u'$ in $\mathsf{Frak}(\mathcal{U}/W)$ given by 
\[
\xymatrix{ {V} \ar[r]^u & {U} \ar[rd]^f \ar[d]^p \\ {V'} \ar[u]^q \ar[r]_{u'} & {U'} \ar[r]_{f'} & {W.}}
\]
is sent under $\Phi_{A,B,!}$ to $\delta^r_q\delta_p^l=\delta_p^l \delta_q^r$. 
\end{enumerate}
\end{lemma} 
\begin{proof}
To verify the first claim we compute for $M\in \Bimod_{\cal{U}}(\AAA,\BBB)$
\begin{align*}
\Hom(\Phi_{A,B,!}(P_u),M) &=\Hom(P_u,
\Phi^\ast_{A,B}M)\\
&=(\Phi^\ast_{A,B}M)(u)\\
&=M_u((fu)^\ast B,  f^\ast A)\qed
\end{align*}
The second claim is verified in a similar way. 
\def\qed{}\end{proof}

\begin{remark} \label{ref-3.7-19} There is an alternative way to think about the projective
objects $P_{B,u,A}$. Introduce the $k$-linear category $\mathfrak{t} = \AAA^{\circ} \otimes_{\uuu} \BBB$ (see \cite{lowen8}, Definition 2.10) as 
follows 
\[
\Ob(\mathfrak{t})=\{(B,u,A)\mid u:V\rightarrow U\in \mathcal{U}, B\in \Ob(\BBB_V), A\in \Ob(\AAA_U)\}
\]
with 
\[
\TTT((B', u', A'), (B, u, A)) = \bigoplus_{u' = puq} \BBB_q(B',
B)\otimes \AAA_p(A,A') 
\]
with obvious composition. 

Then there is an isomorphism of categories 
\begin{equation}
\label{ref-3.7-20}
\Bimod_{\mathcal{U}}(\AAA,\BBB)\lra\Mod(\mathfrak{t}):
M\mapsto ((B, u, A)\mapsto M_u(B,A))
\end{equation}
Under this isomorphism we have
\[
P_{B,u,A}=\mathfrak{t}(-,(B, u, A))
\]
If $W\in \mathcal{U}$, $A,B\in \BBB_W$ then we have an associated functor
\[
\mathsf{Fact}(\mathcal{U}/W)\lra \mathfrak{t}^\circ:(V\xrightarrow{u} U\xrightarrow{f} W)\mapsto ((fu)^\ast B,u, f^\ast A)
\]
Hence we get a corresponding dual functor
\[
\Mod(\mathfrak{t})\lra \mathsf{Nat}(\mathcal{U}/W)
\]
Composing this with \eqref{ref-3.7-20} we obtain a functor
\[
\Bimod_{\mathcal{U}}(\AAA,\BBB)\lra \Nat(\mathcal{U}/W)
\]
which turns out to be precisely $\Phi^\ast_{A,B}$. From this one
easily obtains Lemma \ref{ref-3.6-18}. 
\end{remark}

\subsection{Projective fibered bimodules}
\label{ref-3.8-21}
Let $\aaa$, $\bbb$ be $k$-linear prestacks on $\uuu$ with associated fibered
categories $\AAA$, $\BBB$.
If $W\in \uuu$ and $A\in \mathfrak{a}_W$, $B\in
\mathfrak{b}_W$ the functor
\[
\Bimod^{\text{fib}}_{\cal{U}}(\AAA,\BBB)\lra \Mod(k):M\mapsto M_{1_W}(B,A)
\]
is representable by a projective object
$P^{\text{fib}}_{B,W,A}$. Again concretely for $u':V\lra U$, $B'\in
\BBB_V$, $A'\in \AAA_U$ we have
\[
%(P_{B,W,A}^{\text{fib}})_{u'}(B',A')=\bigoplus_{u:V\rightarrow W}
%\BBB_{u}(B',B)\otimes \AAA_{u'}(u^\ast A,A')
\bigoplus_{u:V\rightarrow W}
\BBB_{1_V}(B',u^\ast B)\otimes \AAA_{u'}(u^\ast A,A')
\]

\begin{remark}\label{remprojfib}
Again, there is an alternative way to think about the projective objects $P^{\text{fib}}_{B,W,A}$. Let $\RRR$ be the underlying linear category of the fibered graded category associated to $\aaa \otimes \bbb$. Concretely,
$$\Ob(\RRR) = \{ (B,W,A) \,\, |\,\, V \in \Ob(\uuu), B \in \Ob(\bbb(U)), A \in \Ob(\aaa(U))\}$$
and
$$\RRR((B',W',A'), (B,W,A)) = \bigoplus_{w: W' \lra W}\aaa(W)(w^{\ast}A, A') \otimes \bbb(W)(B', w^{\ast}B).$$
Then we have an isomorphism of categories
\[
\Bimod(\aaa,\bbb) \cong \Mod(\RRR): M \mapsto ((B,W,A) \mapsto M^W(B,A))
\]
and, by Proposition \ref{ref-3.4-10}, an equivalence of categories
\[
{\Bimod}(\aaa , \bbb) \lra \mathsf{Bimod}^{\mathrm{fib}}_{\uuu}(\AAA,\BBB).
\]
Then $P^{\text{fib}}_{B,W,A}$ is the image of $\RRR(-, (B,W,A))$ under these functors.
\end{remark}

\begin{remark}\label{remunderlin}
With $\TTT$ as in Remark \ref{ref-3.7-19} and $\RRR$ as in Remark \ref{remprojfib}, the functor
$$\Pi^{\ast}: \Bimod(\aaa, \bbb) \lra \Bimod_{\uuu}(\AAA, \BBB)$$ of \S \ref{ref-3.4-7} is induced by an underlying linear functor
$$\Pi: \TTT \lra \RRR: (B, u: V \lra U, A) \longmapsto (B,V, u^{\ast} A).$$
\end{remark}

\subsection{Cohomology of natural systems and the bar complex}
\label{ref-3.9-22}

In \cite{baueswirsching}, the cohomology of $\uuu$ with values in a
natural system $N$ 
 has been defined via a certain complex $\CC(\uuu,
N)$.  This complex computes in fact
\begin{equation}
\label{ref-3.8-23}
\RHom_{\Nat(\mathcal{U})}(\underline{k},N)
\end{equation}
where $\underline{k}$ is the constant natural system with value $k$
(see \cite[Thm.\ 4.4]{baueswirsching}).  If
$N$ is obtained from a presheaf then this reduces to ordinary presheaf
cohomology. I.e.
\begin{proposition}\cite[Proposition 8.5]{baueswirsching}
\label{ref-3.8-24}
Let $F\in \Pre(\mathcal{U})$. Then 
\[
\CC(\uuu, I\!F) \cong\RHom_{\Pre(\mathcal{U})}(\underline{k},F)
\]
\end{proposition}
To compute \eqref{ref-3.8-23} one uses the bar resolution
$B(\underline{k})$ (see the proof of \cite[Thm.\ 4.4]{baueswirsching})
of $\underline{k}$. To define this
let $N = N(\uuu)$ denote the simplicial nerve of $\uuu$. For $v \in
N_n$ given by
\[
\xymatrix{{V_0} \ar[r]_{v_0} & {V_1} \ar[r]_{v_1}& {\dots} \ar[r] & {V_{n-1}} \ar[r]_{v_{n-1}} & {V_n,}}
\]
we put $|v| = v_{n-1}\dots v_1v_0$. For $V\in N_0$ we put $|V|=1_V$.  
Then we have 
\begin{equation}
\label{ref-3.9-25}
B(\underline{k})_n = \oplus_{v \in N_n}P_{|v|}
\end{equation}
where $P_u$ stands for the bimodule associated to the projective
natural system $k\mathsf{Fact}(\mathcal{U})(u,-)$. 

As a complex $B(\underline{k})_n$ is the chain complex of a simplicial
object in $\Nat(\mathcal{U})$ which we will
denote by $\mathcal{B}(\underline{k})$. The degeneracies (which we will not
use) are obtained from the identity maps
\[
\sigma_i^v:P_{|v|}\lra P_{|\sigma_i(v)|}=P_{|v|}
\]
The boundary maps $\partial_i$ are obtained from maps
\[
\partial_i^v: P_{|v|} \lra P_{|\partial_i(v)|}
\]
which we now define. For $i$ different from $0$ and $n$, $|\partial_i(v)| = |v|$, and the
map $\partial_i^v$ is the identity. For $i = 0$ and $i = n$,
$\partial_0^v$ and $\partial_n^v$ are obtained from the following
maps $\partial_i(v)\rightarrow v$ in $\mathsf{Fact}(\mathcal{U})$
(using the contravariant dependence of $P_u$ on $u$). 
\[
\xymatrix{{V_{n-1}} \ar[r]^{|\partial_0(v)|} & {V_n} \ar[d]^1 \\ {V_0} \ar[u]^{v_0} \ar[r]_{|v|} & {V_n}} \hspace{1cm} \text{and} \hspace{1cm} \xymatrix{{V_{0}} \ar[r]^{|\partial_n(v)|} & {V_{n-1}} \ar[d]^{v_{n-1}} \\ {V_0} \ar[u]^{1} \ar[r]_{|v|} & {V_n}}
\]

\begin{remark}
If we consider natural systems as a special instance of bimodules as in Remark \ref{remnatbimod}, then the complex $\CC(\uuu, N)$ is readily seen to coincide with the Hochschild complex defined in \cite{lowen8}.
\end{remark}

\subsection{Resolutions}
Now we return to our standard settings: $\AAA$, $\BBB$ are fibered
$\uuu$-graded categories equipped with a normalized cleavage, $W\in \mathcal{U}$
and $A\in \AAA_W$, $B\in \BBB_W$.  \emph{The following is our key
technical result.}
\begin{lemma}
\label{ref-3.9-26}
Consider the constant natural system $\underline{k}$ on $\uuu/W$. We have
$$L\Phi_{A,B, !}(\underline{k}) = \Phi_{A,B,!}(\underline{k}) =  P^{\mathrm{fib}}_{B,W,A}.$$
\end{lemma}

\begin{proof}
  By definition, $L\Phi_{A,B, !}(\underline{k}) = \Phi_{A,B,
    !}(B(\underline{k}))$. We will define an augmentation map
  $\epsilon: \Phi_{A,B,!}(\mathcal{B}(\underline{k}))_0) \lra
  P^{\mathrm{fib}}_{B,W,A}$ and for every $u':V'\rightarrow U'$,
  $B'\in \BBB_{V'}$, $A'\in \AAA_{U'}$ we will show that the augmented
  simplicial abelian group
  $(\Phi_{A,B,!}(\mathcal{B}(\underline{k}))_{u'}(B',A'),\epsilon)$ is
  contractible. As the chain complex associated to
  $\Phi_{A,B,!}(\mathcal{B}(\underline{k}))$ is
  $\Phi_{A,B, !}(B(\underline{k}))$ this proves the claim.

\medskip

We need to give a detailed description 
of the simplicial object $\Phi_{A,B,!}(\mathcal{B}(\underline{k}))$.
Put
$N=N(\mathcal{U}/W)$. We write the elements of $N_n$ as $(v,f)$ where $v$
represents an element of $N(\mathcal{U})_n$
\[
\xymatrix{{V_0} \ar[r]_{v_0} & {V_1} \ar[r]_{v_1}& {\dots} \ar[r] & {V_{n-1}} \ar[r]_{v_{n-1}} & {V_n,}}
\]
and $f$ is map $V_{n} \rightarrow W$. Then
\[
\sigma_i(v,f)=(\sigma_{i}(v),f)
\]
and 
\[
\partial_i(v,f)=
\begin{cases}
(\partial_i v,f)& \text{if $i\neq n$}\\
(\partial_n v,v_{n-1}f)& \text{if $i=n$}\\
\end{cases}
\]
Combining \eqref{ref-3.9-25} with Lemma \ref{ref-3.6-18}(1) we obtain
\[
\Phi_{A,B,!}(\mathcal{B}(\underline{k}_n))=\bigoplus_{(v,f) \in
  N_n}P_{(f|v|)^\ast B, |v|, f^\ast A}
\]
The explicit description of $\mathcal{B}(\underline{k})$ given in \S\ref{ref-3.9-22} combined with Lemma \ref{ref-3.6-18}(2) allows us to describe the
degeneracies and boundary maps in $\Phi_{A,B,!}(\mathcal{B}(\underline{k}))$.

The degeneracies on $\Phi_{A,B,!}(\mathcal{B}(\underline{k}))$ (which
we do not use) are obtained from the identity maps
\[
\sigma_i^v:P_{(f|v|)^\ast B, |v|, f^\ast A} \lra P_{(f|\sigma_i(v)|)^\ast B,
|\sigma_i(v)|,f^\ast A}=P_{(f|v|)^\ast B,
|v|,f^\ast A}
\]
Likewise the boundary maps $\partial_i$ for $i\neq 0,n$ are obtained
from the identity maps
\[
\partial_i^v: P_{(f|v|)^\ast B, |v|, f^\ast A} \lra 
 P_{(f|\partial_i(v)|)^\ast B, |\partial_i(v)|, f^\ast A}
=P_{(f|v|)^\ast B, |v|, f^\ast A}
\]
$\partial_0$ is obtained from maps
\[
\partial_0^v: P_{(f|v|)^\ast B, |v|, f^\ast A} \lra P_{(f|\partial_0 v|)^\ast B, |\partial_0 v|, f^\ast A}.
\]
where with the notation of \eqref{ref-3.6-17} we have $\partial^v_0=\delta^r_{v_0}$. 

Similarly $\partial_n$ is obtained from maps
\[
\partial_n^v: P_{(f|v|)^\ast B, |v|, f^\ast A} \lra P_{(fv_{n-1}|\partial_n v|)^\ast B, |\partial_n v|, (fv_{n-1})^\ast A}.
\]
where with the notation of \eqref{ref-3.6-17} we have
$\partial^v_n=\delta^l_{v_{n-1}}$.

Now we discuss the case $n=1$. The boundary maps 
\[
\partial_0,\partial_1:\mathcal{B}(\underline{k})_1 \lra \mathcal{B}(\underline{k})_0
\]
are obtained from
\[
\partial^{u}_0: P_{(fu)^{\ast}B, u, f^{\ast}A} \lra P_{f^{\ast}B, 1_U, f^{\ast}A}
\]
\[
\partial^{u}_1: P_{(fu)^{\ast}B, u, f^{\ast}A} \lra P_{(fu)^{\ast}B, 1_V, (fu)^{\ast}A}
\]
These are respectively given by $\delta^r_u$ and $\delta^l_u$.

Let $M$ be fibered bimodule and fix an element $x\in M_{1_W}(B,A)$.
Then since $M$ corresponds to a bimodule over the corresponding
$k$-linear prestack on $\mathcal{U}$ we have corresponding restricted
elements $\rho^f(x)\in M_{1_U}(f^\ast B,f^\ast A)$ and
$\rho^{fu}(x)\in M_{1_V}((fu)^\ast B, (fu)^\ast A)$. 
One checks that the following identity holds in $M_u((fu)^\ast B,f^\ast A)$
\begin{equation}
\label{ref-3.10-27}
\rho^f(x)\cdot \delta_{u,f^\ast B}=\delta_{u,f^\ast A} \cdot \rho^{fu}(x)
\end{equation}
Let 
\[
\epsilon^f_x:P_{f^\ast A,1_U,f^\ast B}\lra M
\]
correspond to $\rho^f(x)\in M_{1_U}(B,A)$. Then \eqref{ref-3.10-27} combined
with the discussion in \S\ref{ref-3.7-15}
yields a commutative diagram
\begin{equation}
\label{ref-3.11-28}
\xymatrix{
& P_{f^\ast B,1_U,f^\ast A}\ar[rd]^{\epsilon^f_x}&\\
P_{(fu)^\ast B,u,f^\ast A}\ar[ru]^{\partial^u_0}\ar[rd]_{\partial^u_1}&& M\\
& P_{(fu)^\ast B,1_V,(fu)^\ast A}\ar[ru]_{\epsilon^{fu}_x}&
}
\end{equation}
We use this to define an augmentation
\[
\epsilon: \Phi_{A,B,!}(\cal{B}(\underline{k})_0) \lra P^{\mathrm{fib}}_{B,W,A}.
\]
For every $f: U \lra W$ in $N_0(\uuu/W)$, we have to define a map
\[
\epsilon^f: P_{f^{\ast}B, 1_U, f^{\ast}A} \lra P_{B,W,A}^{\mathrm{fib}},
\]
Let $M=P^{\text{fib}}_{A,W,B}$. Then 
\[
P^{\text{fib}}_{A,W,B}(A,1_W,B)=
\Bimod_{\mathcal{U}}(\AAA,\BBB)(P^{\text{fib}}_{B,W,A},P^{\text{fib}}_{A,W,B})
\]
contains a canonical element $x$ given by the identity morphism. 
We put $\epsilon^f=\epsilon^f_x$. Then by \eqref{ref-3.11-28} we obtain
\[
\epsilon^f\partial_0^u =\epsilon^{fu}\partial_1^u
\]
Now we show that for every $u':V'\rightarrow U'$ in $\mathcal{U}$ and
for every $B'\in \mathfrak{b}_{V'}$, $A'\in \mathfrak{a}_{U'}$ the
augmented simplicial object
$S\overset{\text{def}}{=}
(\Phi_{A,B,!}(\mathcal{B}(\underline{k}))_{u'}(B',A'),\epsilon)$
is left contractible in the sense of \cite[\S 8.4.6]{Weibel}. That is
we have to produce contracting homotopies 
\[
h_n:S_n\lra S_{n+1}
\]
 for $n\ge -1$ such that
\begin{equation}
\label{ref-3.12-29}
\begin{aligned}
\partial_0 h_n&=\operatorname{Id}_{S_n} \\
\partial_i h_n&=h_{n-1} \partial_{i-1} \qquad \text{($i>0$)}
\end{aligned}
\end{equation}
where $\partial_0:S_0\rightarrow S_{-1}$ should be interpreted as $\epsilon$. 

We start by defining   a map
\[
h_{-1}: S_{-1}=(P^{\mathrm{fib}}_{B,W,A})_{u'}(B',A') \lra
\bigoplus_{f:U\rightarrow W \in N_0}(P_{f^{\ast}B, 1_{U}, f^{\ast}A})_{u'}(B',A')
=S_0
\]
We have (see \S\ref{ref-3.8-21})
\[
(P^{\mathrm{fib}}_{B,W,A})_{u'}(B',A'):\bigoplus_{f:V'\rightarrow W}
\BBB_{1_{V'}}(B',f^\ast B)\otimes \AAA_{u'}(f^\ast A,A')
\]
and
\[
(P_{f^{\ast}B, 1_{U}, f^{\ast}A})_{u'}(B',A')=\bigoplus_{pq=u'}\BBB_q(B',f^\ast B)
\otimes \AAA_p(f^\ast A, A')
\]
We send the part of $(P^{\text{fib}}_{B,W,A})_{u'}(B',A')$ given by
\[
\BBB_{1_{V'}}(B',f^\ast B)\otimes \AAA_{u'}(f^\ast A,A')
\]
via  the identity
morphism to the identical part of $P_{f^{\ast}B, 1_{V'}, f^{\ast}A}$ 
corresponding to $q=1_{V'}$, $p=u'$.

Before continuing we will check
\begin{equation}
\label{ref-3.13-30}
\epsilon h_{-1}=\operatorname{Id}_{S_{-1}}
\end{equation}
To do this we need to make explicit $\epsilon:S_0\rightarrow S_{-1}$. One
verifies that $\epsilon$ sends an element
\[
b\otimes a \in \BBB_q(B',f^\ast B)
\otimes \AAA_p(f^\ast A, A')\subset (P_{f^{\ast}B, 1_{U}, f^{\ast}A})_{u'}(B',A')
\]
to the element of
\[
\BBB_{1_{V'}}(B',(fq)^\ast B)
\otimes \AAA_{u'}((fq)^\ast A, A') \subset (P^{\mathrm{fib}}_{B,W,A})_{u'}(B',A')
\]
 given by
\[
\delta^{-1}_{q,f^\ast B}\cdot b\otimes a\cdot \delta_{q,f^\ast A} 
\]
where we have committed a slight abuse of notation (cartesian arrows
are not invertible but left multiplying with them is).  So if $q=1_{V'}$
then $\epsilon(b\otimes a)=b\otimes a$. From this one deduces immediately
$\epsilon h_{-1}=\operatorname{Id}_{S_{-1}}$. 

For use below we also compute the other composition
$h_{-1}\epsilon$. Let $b\otimes a$ be as above. Then
\begin{equation}
\label{ref-3.14-31}
\begin{aligned}
  h_{-1}\epsilon(b\otimes a)&=h_{-1}(\delta^{-1}_{q,f^\ast B}\cdot
  b\otimes a\cdot \delta_{q,f^\ast A} )\\
&=\delta^{-1}_{q,f^\ast B}\cdot
  b\otimes a\cdot \delta_{q,f^\ast A}
\end{aligned}
\end{equation}
Next we define $h_n$. This is a map
\[
h_{n}: \bigoplus_{(u,f) \in N_n}(P_{(f|u|)^\ast B, |u|, f^\ast A})_{u'}(B',A')\lra
\bigoplus_{(u,f) \in N_{n+1}}(P_{(f|u|)^\ast B, |u|, f^\ast A})_{u'}(B',A')
\]
The summand on the left  attached to 
\[
(u,f) = (U_0 \lra \dots \lra
U_n\lra W) \in N_n(\uuu/W)
\] 
is given by
\[
\bigoplus_{u'=p|u|q}\BBB_q(B',(f|u|)^\ast B)\otimes \AAA_p(f^\ast A,A')
\]
We send an element $b\otimes a\in \BBB_q(B',(f|u|)^\ast B)\otimes
\AAA_p(f^\ast A,A')$ to the element of
\[
\BBB_{1_{V'}}(B',(f|uq|)^\ast B)\otimes \AAA_p(f^\ast A,A')
\subset (P_{(f|uq|)^\ast B, |uq|, f^\ast A})_{u'}(B',A') 
\]
given by $\delta_{q,(f|u|)^\ast B}^{-1}b\otimes a$. Here
$(P_{(f|uq|)^\ast B, |uq|, f^\ast A})_{u'}$ is considered as being attached to
\[
(\bullet \xrightarrow{q} U_0 \lra \dots \lra
U_n\lra W) \in N_{n+1}(\uuu/W)
\]

We now check that $h$ has the required properties. Let $b\otimes a$
be as in the previous paragraph.  We compute for $n\ge 0$
\begin{equation}
\label{ref-3.15-32}
\begin{aligned}
\partial_0 h_n(b\otimes a)
&=\delta^r_q(\delta_{q,(f|u|)^\ast
  B}^{-1}b\otimes a)\\ 
& =\delta_{q,(f|u|)^\ast
  B}\delta_{q,(f|u|)^\ast
  B}^{-1}b\otimes a\\
&=b\otimes a
\end{aligned}
\end{equation}
If $0<i<n+1$ then
\begin{equation}
\label{ref-3.16-33}
\begin{aligned}
\partial_i h_n(b\otimes a)&=\delta_{q,(f|u|)^\ast
  B}^{-1}b\otimes a
\end{aligned}
\end{equation}
and if $i=n+1$, $n>0$ then
\begin{equation}
\label{ref-3.17-34}
\begin{aligned}
\partial_{n+1} h_n(b\otimes a)&=\delta^l_{u_{n-1}}(\delta_{q,(f|u|)^\ast
  B}^{-1}b\otimes a)\\
&=\delta_{q,(f|u|)^\ast
  B}^{-1}b\otimes a \delta_{u_{n-1},f^\ast A}
\end{aligned}
\end{equation}
If $n=0$, $i=1$ then the same computation yields
\begin{equation}
\label{ref-3.18-35}
\partial_{1} h_0(b\otimes a)=\delta_{q,f^\ast
  B}^{-1}b\otimes a \delta_{q,f^\ast A}
\end{equation}
Now assume $n>0$. We compute
\begin{equation}
\label{ref-3.19-36}
\begin{aligned}
 h_{n-1}\partial_0(b\otimes a)&=h_{n}\delta^r_{u_0}(b\otimes a)\\
&=h_{n}(\delta_{u_0,(f| \partial_0 u|)^\ast B}b\otimes a)\\
&=\delta_{u_0q,(f|\partial_0 u|)^\ast
  B}^{-1}\delta_{u_0,(f|\partial_0 u|)^\ast B}b\otimes a\\
&=\delta_{q,(f|u|)^\ast
  B}^{-1}b\otimes a
\end{aligned}
\end{equation}
and for $0<i<n$
\begin{equation}
\label{ref-3.20-37}
\begin{aligned}
 h_{n-1}\partial_i(b\otimes a)&=h_{n-1}(b\otimes a)\\
&=\delta_{q,(f|u|)^\ast
  B}^{-1}b\otimes a
\end{aligned}
\end{equation}
Finally
\begin{equation}
\label{ref-3.21-38}
\begin{aligned}
h_{n-1}\partial_n(b\otimes a)&=h_n\delta^l_{u_{n-1}}(b\otimes a)\\
&=h_n(b\otimes a\delta_{u_{n-1},f^\ast A})\\
&=\delta_{q,(f|u|)^\ast
  B}^{-1}b\otimes a\delta_{u_{n-1},f^\ast A}
\end{aligned}
\end{equation}
The conditions \eqref{ref-3.12-29} now follow by combining
\eqref{ref-3.13-30} \eqref{ref-3.14-31} \eqref{ref-3.15-32} \eqref{ref-3.16-33} \eqref{ref-3.17-34} \eqref{ref-3.18-35} \eqref{ref-3.19-36} \eqref{ref-3.20-37} \eqref{ref-3.21-38}.
\end{proof}

\section{The Cohomology Comparison Theorem}\label{ref-4-39}

\subsection{Main result}
Let $\aaa$ and $\bbb$ be $k$-linear prestacks on $\mathcal{U}$ with
associated fibered graded categories $\AAA$ and $\BBB$. The following
is our main result.
\begin{theorem}\label{ref-4.1-40}
  The functor $\Pi^{\ast}: \Bimod(\aaa,\bbb) \lra
  \mathsf{Bimod}_{\uuu}(\AAA, \BBB)$ induces a fully faithful functor
\[
\Pi^{\ast}: D(\Bimod(\aaa,\bbb)) \lra D(\mathsf{Bimod}_{\uuu}(\AAA,
\BBB))
\]
In particular, for $M, N \in
\Bimod(\aaa,\bbb)$, there are isomorphisms
$$\Ext^i_{\Bimod(\aaa,\bbb)}(M,N) \cong \Ext^i_{\mathsf{Bimod}_{\uuu}(\AAA, \BBB)}(\Pi^{\ast}M, \Pi^{\ast}N)$$
for all $i$.
\end{theorem}

\begin{proof}
  According to Proposition \ref{ref-3.4-10} we may replace
  $\mathsf{Bimod}(\aaa, \bbb)$ by
  $\Bimod^{\text{fib}}_{\mathcal{U}}(\AAA,\BBB)$. As before we denote the
  inclusion $\Bimod^{\text{fib}}_{\mathcal{U}}(\AAA,\BBB)\rightarrow
  \mathsf{Bimod}_{\uuu}(\AAA, \BBB)$ by $I$.

For all objects $M,N\in D(\Bimod^{\text{fib}}_{\mathcal{U}}(\AAA,\BBB))$ we have
to prove that the canonical map 
\begin{equation}
\label{ref-4.1-41}
\RHom_{\Bimod^{\text{fib}}_{\mathcal{U}}(\AAA,\BBB)}(M,N)
\lra
\RHom_{\Bimod_{\mathcal{U}}(\AAA,\BBB)}(I\!M,I\!N)
\end{equation}
is an isomorphism. We claim that it is sufficient to check this for
$M=P^{\text{fib}}_{B,W,A}$. 

To see this note that the projective objects $P^{\text{fib}}_{A,W,B}$
form a system of compact generators for
$D(\Bimod^{\text{fib}}_{\mathcal{U}}(\AAA,\BBB))$.  If we fix $N$ then
the $M$ for which \eqref{ref-4.1-41} is an isomorphism is a triangulated
subcategory of $D(\Bimod^{\text{fib}}(\AAA,\BBB))$. If this subcategory
contains the generators $P^{\text{fib}}_{B,W,A}$ then it is
everything.

We now assume $M=P^{\text{fib}}_{B,W,A}$. Then we have
\[
\RHom_{\Bimod^{\text{fib}}_{\mathcal{U}}(\AAA,\BBB)}(P^{\text{fib}}_{B,W,A},N)=
N_{1_W}(B,A)
\]
On the other hand we have
\begin{align*}
\RHom_{\Bimod_{\mathcal{U}}(\AAA,\BBB)}(I\!P^{\text{fib}}_{A,W,B},I\!N)&=
\RHom_{\Bimod_{\mathcal{U}}(\AAA,\BBB)}(L\Phi_{A,B,!}\underline{k},I\!N)&& \text{(Lemma \ref{ref-3.9-26})}\\
&=\RHom_{\Nat(\mathcal{U}/W)}(\underline{k},\Phi^\ast_{A,B}I\!N)&&\text{(Adjointness)}\\
&=\RHom_{\Nat(\mathcal{U}/W)}(\underline{k},I\Psi^\ast_{A,B}\!N)&&\text{(Prop.\ \ref{ref-3.5-14})}\\
&=\RHom_{\Pre(\mathcal{U}/W)}(\underline{k},\Psi^\ast_{A,B}N)&& \text{(Prop.\ \ref{ref-3.8-24})}\\
&=(\Psi^\ast_{A,B}N)(1_W) &&\text{($1_W$ is the final object of $\mathcal{U}/W$)}\\
&=N_{1_W}(B,A)&&\text{\eqref{ref-3.3-12}}
\end{align*}
which is the same. 
\end{proof}
The following corollary was
announced in \cite{lowen8}:
\begin{corollary}
  Let $\aaa$ be a $k$-flat $k$-linear prestack on $\mathcal{U}$ with
  associated fibered  category $\AAA$, and let $\CC(\aaa) =
  \CC(\AAA)$ be the Hochschild complex of $\aaa$ defined in
  \cite{lowen8}. Then there is a quasi-isomorphism
\[
\CC(\aaa) \cong \RHom_{\Bimod(\aaa)}(\aaa, \aaa).
\]
\end{corollary}

\begin{proof}
It was shown in \cite[Proposition 3.13]{lowen8} that $$\CC(\AAA) \cong \RHom_{\mathsf{Bimod}_{\uuu}(\AAA)}(\AAA, \AAA).$$
To obtain the desired result, it suffices to apply Theorem \ref{ref-4.1-40} with $\BBB = \AAA$ and note that $\Pi^{\ast}(\aaa) = \AAA$.
\end{proof}

\section{The Special Cohomology Comparison Theorem revisited}\label{ref-5-42}
In this section, we deduce
the original Special Cohomology Comparison Theorem~\cite{gerstenhaberschack2} from Theorem \ref{ref-4.1-40}.

If $\AAA$ is a $k$-linear category then we define $\bar{\AAA}$ to be the
endomorphism ring of the generator $P=\oplus_A \mathfrak{a}(-,A)$ of
$\Mod(\mathfrak{a})$. I.e. by Yoneda
\[
\bar{\AAA}=\prod_{B \in \AAA} \bigoplus_{A \in \AAA} \AAA(B,A).
\]
Elements of $\bar{\AAA}$ can be represented by column-finite matrices with (row, column) indices given by $(A,B)$. We denote the idempotent corresponding to $1_A \in \AAA(A,A)$ by $e_A$. 
\begin{lemma} \label{ref-5.1-43} Let $\AAA$ and $\BBB$ be linear categories. Then the functor
\[
D(\Bimod(\mathfrak{a},\mathfrak{b})) \lra D(\Bimod(\bar{\mathfrak{a}},
\bar{\mathfrak{b}})):M\mapsto\prod_{B \in \AAA} \bigoplus_{A \in \AAA} M(B,A)
\]
is fully faithful.
\end{lemma}
\begin{proof} This follows from the fact that a left inverse is given by
\[
N\mapsto ((B,A)\mapsto e_A N e_B)\qed
\]
\def\qed{}\end{proof}
Let $\mathcal{U}$ be a poset.
If $\mathfrak{a}$ is $\mathcal{U}$-graded category then we denote by
$\tilde{\AAA}$ be 
underlying $k$-linear category of $\AAA$. More precisely
\[
\tilde{\AAA}(B_V,A_U) = \begin{cases} 
\AAA_{V\leq U}(B_V, A_U) &\text{if} \,\, V \leq U\\ 
0 &\text{otherwise.} \end{cases}
\]
for $B_V\in \BBB_V$, $A_U\in \AAA_U$. 
\begin{lemma}\label{ref-5.2-44} The functor
\[
D(\Bimod_{\mathcal{U}}(\AAA,\BBB)) \lra 
D(\Bimod(\tilde{\AAA},\tilde{\BBB})):N \lra \tilde{N}
\]
with
\[
\tilde{N}(B_V,A_U)=\begin{cases} N_{V\leq U}(B_V,A_U)
&\text{if} \,\, V \leq U \\ 0 & \text{otherwise.}
\end{cases}
\]
for  $B_V\in \BBB_V$, $A_U\in \AAA_U$ is fully faithful.
%which sends an object in $D(\Bimod_{\mathcal{U}}(\AAA,\BBB))$ to the
%object in $D(\Bimod(\AAA,\BBB))$ taking the same values on
%$\Ob(\mathfrak{\AAA})\times \Ob(\mathfrak{\BBB})$, is fully faithful.
\end{lemma}
\begin{proof} %To define a left inverse we simply send
%$N\in D(\Bimod(\tilde{\AAA},\tilde{\BBB}))$ to the 
  A left inverse is given by associating to $N\in
  D(\Bimod(\tilde{\AAA},\tilde{\BBB}))$ the object $\check{N}$ in
  $D(\Bimod_{\mathcal{U}}(\AAA,\BBB))$ such that 
$\check{N}_{V\le U}(B_V,A_U)=N(B_V,A_U)$.
In other words the left inverse simply forgets the values of $N(B_V,A_U)$ for $V\not\leq U$.
%This follows from the fact that a left
%inverse is given by
%\[
%N\mapsto \left[(B,A)\mapsto \begin{cases} N_{V\leq U}(B,A)
%&\text{if} \,\, V \leq U \\ 0 & \text{otherwise.}
%\end{cases}\right] \qed
%\]
\def\qed{}\end{proof}
If $\aaa$ is a $k$-linear prestack on $\mathcal{U}$ then we write
$\aaa!=\bar{\tilde{\AAA}}$ where $\mathfrak{\AAA}$ is the 
$\mathcal{U}$-graded category associated to $\mathcal{A}$. Concretely
\[
\mathcal{A}!=\prod_{V \in \uuu, A_V \in \aaa(V)} \bigoplus_{V \leq U, A_U \in \aaa(U)} \mathcal{A}(V)(A_V, A_U|_V)
\]
\begin{theorem}\label{ref-5.3-45}
  Let $\aaa$ and $\bbb$ be $k$-linear prestacks on a poset $\uuu$.
  The canonical functor
\[
(-)!: \mathsf{Bimod}(\aaa, \bbb) \lra 
\mathsf{Bimod}(\mathcal{A}!, \mathcal{B}!): M \lra M!
\]
with
\[
M! = \prod_{V \in \uuu, B_V \in \bbb(V)} \bigoplus_{V \leq U, A \in \aaa(U)} M(V)(B_V, A_U|_V)
\]
induces a fully faithful functor between the respective derived categories.
\end{theorem}
\begin{proof}
Let $\AAA$, $\BBB$ be the $\mathcal{U}$-graded categories associated to $\aaa$,
$\bbb$. 

The functor $D(-)!$ is the composition of the following fully faithful
embeddings
\begin{multline*}
D(\mathsf{Bimod}(\aaa, \bbb))
\xrightarrow{\text{Thm.\ \ref{ref-4.1-40}}}
%D(\mathsf{Bimod}^{\text{fib}}_{\mathcal{U}}(\AAA, \BBB))
%\xrightarrow{\text{Lem.\ \ref{ref-3.9-26}}}
 D(\mathsf{Bimod}_{\mathcal{U}}(\AAA, \BBB))\\
\xrightarrow{\text{Lem.\ \ref{ref-5.2-44}}}
D(\mathsf{Bimod}(\tilde{\AAA}, \tilde{\BBB}))\xrightarrow{\text{Lem.\ \ref{ref-5.1-43}}} D(\mathsf{Bimod}(\bar{\tilde{\AAA}}, \bar{\tilde{\BBB}}))\qed
\end{multline*}
\def\qed{}\end{proof}
The following corollary is the
Special Cohomology Comparison Theorem in the sense of Gerstenhaber and Schack. 
\begin{corollary}
 Let $\aaa$ be a presheaf of $k$-algebras on a poset $\uuu$. Put
\[
\mathcal{A}!=\prod_{V \in \uuu} \bigoplus_{V \leq U} \mathcal{A}(V)
\]
The canonical functor 
\[
(-)!: \mathsf{Bimod}(\aaa) \lra 
\mathsf{Bimod}(\mathcal{A}!): M \lra M!
\]
with
\[
M! = \prod_{V \in \uuu} \bigoplus_{V \leq U} M(V)
\]
induces a fully faithful functor between the respective derived categories.
\end{corollary}
\begin{proof} If we view a presheaf as a prestack then the categories
$\mathcal{A}(V)$ contain only one object which is left out of the
notation. This gives the simplified form of $\mathcal{A}!$ and $M!$ is the statement
of this corollary. 
\end{proof}
\appendix
\section{Relation with universal localization}
Let $\AAA$, $\BBB$ be $\mathcal{U}$-graded fibered categories equipped with
a normalized cleavage. 
Let $u:V\rightarrow U$, $p:U\rightarrow U'$,
$B\in \AAA_V$,  $A'\in \AAA_{U'}$. The natural transformation between
the functors
\begin{align*}
M&\mapsto M_u(B,p^\ast A')\\
M&\mapsto M_{pu}(B,A')
\end{align*}
corresponding to $\delta^l_p$ is given by left multiplication by $\delta_{p,A'}$
(see \S\ref{ref-3.7-15}). Hence $M$ is fibered if and only if
for all $u,p,B,A'$ as above $\Bimod_{\mathcal{U}}(\AAA,\BBB)(-,M)$ inverts the
corresponding $\delta^l_p$.  Thus if we denote the collections of such
morphisms by $\Sigma$ then we get
\[
\Bimod^{\text{fib}}_{\mathcal{U}}(\mathfrak{b},\mathfrak{a})=
\{M\in \Bimod_{\mathcal{U}}(\mathfrak{b},\mathfrak{a})\mid \text{$M$ inverts all morphisms in $\Sigma$}\}
\]
Using the description of
$\Bimod_{\mathcal{U}}(\mathfrak{b},\mathfrak{a})$ as a module category 
in Remark \ref{ref-3.7-19}
we also obtain
\[
\Bimod^{\text{fib}}_{\mathcal{U}}(\mathfrak{b},\mathfrak{a})=\Mod(\Sigma^{-1}
\mathfrak{t})
\]
This shows that the relation between
$\Bimod^{\text{fib}}_{\mathcal{U}}(\mathfrak{b},\mathfrak{a})$ and
$\Bimod_{\mathcal{U}}(\mathfrak{b},\mathfrak{a})$ is controlled by a universal
localization (which is given by $\Pi: \TTT \lra \RRR$ as in Remark \ref{remunderlin}). Unfortunately this is not an Ore localization. The
homological behavior of a general universal localization seems difficult
to understand. See e.g.\ \cite{ARS,schofield} for partial
results.

Theorem \ref{ref-4.1-40} may be interpreted as saying
that $\Sigma^{-1} \mathfrak{t}/\mathfrak{t}$ is \emph{stably flat}
(see \cite{NR}). This implies for example that the $K$-theory of
the categories of bimodules and fibered bimodules are related by a long exact sequence
\cite[Thm 0.1]{neeman2}.

\def\cprime{$'$}
\providecommand{\bysame}{\leavevmode\hbox to3em{\hrulefill}\thinspace}
\bibliography{BibCCT}
\bibliographystyle{amsplain}

\end{document}